\documentclass[12pt]{amsart} 
\usepackage{amssymb}
\usepackage{stackrel}
 \calclayout
\theoremstyle{plain}
\newtheorem{thm}{Theorem}[section]

\newtheorem{prop}{Proposition}[section]

\newtheorem{defn}{Definition}[section]
\newtheorem{ex}[thm]{Example}

\newcommand{\begintheorem}{\addtocounter{equation}{1}\begin{theorem}}
\newcommand{\beginlemma}{\addtocounter{equation}{1}\begin{lemma}}
\newcommand{\beginproposition}{\addtocounter{equation}{1}\begin{proposition}}
\newcommand{\begindefinition}{\addtocounter{equation}{1}\begin{definition}}
\newcommand{\begincorollary}{\addtocounter{equation}{1}\begin{corollary}}

\usepackage[bookmarksnumbered, colorlinks, plainpages]{hyperref}
\hypersetup{colorlinks=true,linkcolor=red, anchorcolor=green, citecolor=cyan, urlcolor=red, filecolor=magenta, pdftoolbar=true}
\begin{document} 
%
%
%
%
%
%
%
%
%

\title[ Some applications of uniformly $ p $-convergent sets]
{Some  applications of uniformly $ p $-convergent sets}
\author {M.\ Alikhani.}
\address{Department of Mathematics, University of Isfahan}
\email{m2020alikhani@ yahoo.com}
\subjclass{Primary: 46B25; Secondary: 46E50, 46G05.}
\keywords{ uniformly completely continuous subsets; weakly sequentially continuous differentiable mappings; Dunford-Pettis property of order $ p $;}
\begin{abstract}
 In this paper,  we introduce a new class of subsets of bounded linear operators between Banach spaces  which  is $ p $-version of the uniformly completely continuous sets.\ Then, we study the relationship between these sets with the equicompact sets.\ Moreover, we introduce the concept of weakly $ p $-sequentially continuous differentiable mappings and   obtain some characterizations of these mappings.\ Finally, we
give a  factorization result for differentiable mappings through p-convergent
operators.
\end{abstract}
\maketitle
\section{Introduction}

Let $ X $ and $  Y$ be Banach spaces.\ 
The study uniformly completely continuous sets in the class of all bounded linear operators between Banach
spaces have been obtained in
recent years by several authors.\
The  research works of Cilia et al. \cite{ce2},
shows that if $ U \subseteq X $ is an open convex and
 $ f : U \rightarrow Y $ is a differentiable mapping whose derivative $ f^{\prime}   $ is uniformly
continuous on $ U $-bounded subsets of $ U $ ($f\in  C^{1u} (U,Y)$), then  $ f $
 takes weakly Cauchy  $ U $-bounded
sequences into  norm convergent sequences (in short, $ f\in C_{wsc}(U, Y ) $), 
 if and only if $ f^{\prime} $ takes Rosenthal and $ U $-bounded subsets
of $ U $ into uniformly completely continuous subsets of $  L(X, Y ).$\ 
For more information in this kinds of researches, we refer the reader to \cite{A,ce,ce3,ce4,DM,gg1,m, spd, spd1, spd2} and references therein.\\
Recently,  Chen et al. \cite{ccl1}, by introducing the notion $ p $-$ (V) $ sets, showed that an operator $ T:X\rightarrow Y $  is $ p $-convergent if and only if its adjoint $ T^{\ast} : Y^{\ast}\rightarrow X^{\ast}$
takes bounded subsets of $ Y^{\ast} $ into $p  $-$ (V) $ subset $ X^{\ast}. $\
Motivated by the above work and the research works of Cilia et al \cite{ce2,ce3,ce4}, we give similar
results for differentiable mappings.\ Here,
we  introduce the notion uniformly $ p $-convergent sets and we try answer to the  following interesting  questions:\
\begin{itemize}
\item 
For given a differentiable mapping $  f : U \rightarrow Y, $
under which conditions  its derivative $ f^{\prime} $ takes $  U$-bounded  weakly $ p $-precompact subsets of $  U$ into uniformly
$ p $-convergent subsets?\
\item 
For given a differentiable mapping $  f : U \rightarrow Y, $
under which conditions  its derivative $ f^{\prime} $ takes $  U$-bounded sets into uniformly $ p $-convergent sets?
\end{itemize}

This paper deals with the $ p $-version of uniformly completely continuous sets and weakly sequentially continuous differentiable mappings.\
In Section 2 of this article provides a wide range of definitions
and concepts in Banach spaces.\ These concepts are
mostly well known, but we need them in the sequel.\\
In Section 3, we define the concepts of uniformly $ p $-convergent sets and  weakly $ p $-sequentially continuous differentiable mappings.\  Also, we
apply the notion
 weakly equicompact sets in order to
 find a characterization for those Banach spaces in which the double dual of them
have  the  $p$-Schur property.\
Finally, we find some  equivalent conditions 
 for all  $ f \in C^{1u} (U,Y) $ 
 such that
 $ f^{\prime} $ takes $ U $-bounded and weakly $ p $-precompact subsets of $  U$ into uniformly
$ p $-convergent subsets of the class of all bounded linear operators from $ X $ to $ Y. $\\
 In the Section 4, we
apply the concept uniformly $ p $-convergent sets in order to find a factorization result for a differentiable mapping through
a  $ p $-convergent operator.\ 
\section{ Notions and Definitions}
Throughout this paper $ X,Y $ and $ Z$ will always denote real Banach spaces and  $ U $ is an open convex subset of $ X. $\ We denote the spaces
of all bounded linear operators, compact operators and weakly compact operators from $  X$ into $  Y$ 
by $ L(X, Y ) ,~K(X,Y)$ and $  W(X, Y ) ,$ respectively.\
The topological dual of $  X$ is denoted by $ X^{\ast} $ and
the adjoint of an operator $ T $ is denoted by $ T^{\ast} .$\ 
Also we use $ \langle x^{\ast},x \rangle $ or $ x^{\ast} (x)$ for the duality between $ x\in X$ and $ x^{\ast}\in X^{\ast} .$\ We denote the closed unit ball
of $  X$ and the identity operator on $  X$ by $ B_{X} $ and $ id_{X} $ respectively.\  $ p^{\ast} $ will always denote the conjugate number of $ p $ for
 $ 1\leq p<\infty; $ if $ p=1, $ $ \ell_{p^{\ast}} $ plays the role of $ c_{0}. $\ The unit coordinate vector in $ \ell_{p} $ (resp. $ c_{0} $ or $ \ell_{\infty} $) is denoted by $ e^{p}_{n} $(resp. $ e_{n} $).
In this paper $ 1\leq p<\infty, $ except for the
cases where we consider other assumptions.\\
To state our results, we need to recall some definitions.\
A sequence $ (x_{n})_{n} $ in $ X $ is called weakly $ p $-summable, if $ (x^{\ast}(x_{n}))_{n} \in \ell_{p}$ for each $ x^{\ast}\in X^{\ast} .$\ We denote by $ \ell_{p}^{w}(X) $
 the space
of all weakly $  p$-summable sequences in $ X; $ see \cite{djt}.\ 
 A bounded subset $ K $ of $ X $ is relatively weakly-$ p $-compact, if every sequence in $  K$ has a weakly-$ p $-convergent subsequence
with limit in $  X.$\
 A sequence $(x_{n})_{n}$ in $ X $ is called weakly
$ p $-Cauchy, provided that $ (x_{m_{k}}-x_{n_{k}})_{k}\in \ell_{p}^{w}(X) $  for any increasing sequences $ (m_{k})_{k} $
and $ (n_{k})_{k} $ of positive integers; see \cite{ccl}.\ A subset $ K$ of $ X $ is said to be weakly $ p$-precompact,
provided that every sequence from $ K $ has a weakly $ p$-Cauchy subsequence; see \cite{ccl}.\ Note that the
weakly $ \infty $-precompact sets are precisely the weakly precompact sets or Rosenthal sets.\ An operator $ T \in L(X,Y)$ is said to be
weakly $  p$-precompact, if $ T(B_{X}) $ is weakly $ p $-precompact.\ 
An  operator $ T \in L(X,Y)$ is called $ p$-convergent, if $ \displaystyle\lim_{n\rightarrow \infty} \parallel  T(x_{n})   \parallel=0$ for all $ (x_{n}) _{n}\in \ell_{p}^{w}(X).$\ We denote the space of all
$  p$-convergent operators from $ X $ into $ Y, $   by
$ C_{p}(X,Y) ;$ see \cite{cs}.\  If the identity operator on $ X $ is $ p $-convergent, we say that a Banach space $ X $ has
the $ p $-Schur property, which is equivalent to every weakly $  p$-compact subset of $  X$ is norm compact; see \cite{dm}.\
A Banach space $  X$ is said to have the Dunford-Pettis
property of $  p$ (in short, $ (DPP_{p}) $), if for any Banach space $ Y, $ every weakly compact operator
$ T : X \rightarrow Y $ is $  p$-convergent; see \cite{ccl}.\
A bounded subset $ K $ of $ X^{\ast} $ is a $ p $-$ (V ) $ set, if $ \displaystyle \lim_{n\rightarrow\infty}\displaystyle \sup_{x^{\ast}\in K}\vert x^{\ast}(x_{n}) \vert =0,$ 
for every $ (x_{n})_{n}\in \ell_{p}^{w}(X); $ see  \cite{ccl1}.\ Given $ x, y \in X, $  the
segment with bounds $ x $ and $ y$ denoted by $ I(x, y). $\ 
A set $ B\subset U $  is $ U $-bounded, if it is bounded and the distance 
between $ B$ and the boundary  of $ U$ is strictly positive; see \cite{ce3}.\  
  The space
of all differentiable mappings $ f : U \rightarrow Y $ whose derivative $ f^{\prime} : U \rightarrow L(X, Y ) $ is
uniformly continuous on $  U$-bounded subsets of $ U $ will be denoted by $ C^{1u} (U,Y);$ see \cite{ce2}.\
A set $ M \subset K(X, Y ) $ is equicompact, if there exists a null sequence $ (x^{\ast}_{n})_{n} $
 in $ X^{\ast} $ so that $ \parallel T(x)  \parallel \leq \sup_{n}\vert x^{\ast}_{n}(x)\vert$ for all $ x \in X $ and all $ T \in M,$ which is equivalent to  every bounded sequence $ (x_{n})_{n} $ in
$  	X$ has a subsequence $ (x_{k_{n}})_{n} $ such that $ (Tx_{k_{n}})_{n} $ is uniformly convergent
for $ T\in M; $ see \cite{spd}.\ For given a mapping $ f : U \rightarrow Y $ and a class $  \mathcal{M}$ of subsets of $ U $ such that every
singleton belongs to $  \mathcal{M},$ the mapping $  f$ is $  \mathcal{M}$-differentiable at $ x \in U, $ if there exists
an operator $ f^{\prime}(x) \in L(X,Y)$ such that
$$\lim_{\varepsilon\rightarrow 0}\frac{f(x+\varepsilon y)-f(x)-f^{\prime}(x)(\varepsilon y)}{\varepsilon} =0 $$
uniformly to $ y $ on each member of $ \mathcal{M} .$\ In this case, we write
$ f\in D_{\mathcal{M}}(x, Y) ;$ see \cite{gg1}.\ We say that
a mapping $ f $ is G$ \hat{a} $teaux differentiable at $  x,$ if
$ f\in D_{\mathcal{M}}(x, Y) $ where  $ \mathcal{M} $ is the class of all single-point subsets of $ X. $\ We also, say that $ f $ is Fr\'echet differentiable at $ x $ if $ f\in D_{\mathcal{M}}(x, Y) ,$  where $ \mathcal{M} $ is the
class of all bounded subsets of $ X.$

\section{Weakly $ p $-sequentially continuous differentiable mappings}
Here, we introduce the notion uniformly $ p $-convergent sets in $ L(X,Y) $ and  give some properties of these sets.\ Then,  we study the weakly $ p $-sequentially continuous differentiable mappings.\ 
\begin{defn}\label{d1}
Let $ K \subset L(X, Y ). $\ We say that $ K $ is a uniformly $ p $-convergent set, if  every 
$ (x_{n})_{n}\in \ell_{p}^{w} (X)$   converges uniformly on $ K, $ that is,
$$ \lim _{n} \sup_{T \in K} \Vert T(x_{n})\Vert =0.$$
\end{defn}
Note that, the uniformly $ \infty $-convergent sets in $ L(X,Y) $ are precisely  the uniformly completely
continuous sets; see \cite{spd}.\
Also, every uniformly $ q $-convergent subset of $ L(X,Y) $ is uniformly $ p $-convergent,
whenever $ 1\leq p<q\leq \infty. $\ It would be interesting to obtain conditions under which every uniformly
$ p $-convergent set in $ L(X,Y) $ is uniformly $ q $-convergent.\  In particular,  we obtain a characterization for those Banach spaces in which
uniformly $ p $-convergent sets in $ X^{\ast} $ are uniformly $ q $-convergent; see \cite{afz}.\\
The following example shows that, there exists a uniformly $ p $-convergent subset of $ L(\ell_{2},Y) $ so that 
it is not uniformly $ q $-convergent.
\begin{ex} \label{e1}
Let $ X=\ell_{2} $ and $ Y$ be an arbitrary Banach space.\
Since $ \ell_{2} $ does not have $ 2 $-Schur property, $ B_{\ell_{2}} $ is not a $ 2$-$ (V ) $ set in $ \ell_{2} $.\ Therefore $B_{\ell_{2}} $ 
is not uniformly $ 2 $-convergent subsets of $ \ell_{2} .$\
On the other hand, $ \ell_{2} $ contain no copy of $ c_{0} .$\ Therefore, $ \ell_{2} $ has the $ 1 $-Schur property; see \cite{dm}.\ Hence,
$ B_{\ell_{2}} $ is a $ 1 $-$ (V ) $ set and so, $ B_{\ell_{2}} $ is a uniformly $ 1 $-convergent subsets of $ \ell_{2} .$\ Now, let $ 0\neq y_{0}\in B_{Y} $ and
$ S : \mathbb{R}\rightarrow Y $ be the operator
given by $ S(\lambda) := \lambda y_{0}~(\lambda \in \mathbb{R}). $\ Define an operator $ T : \ell_{2}\rightarrow L(\ell_{2}, Y ) $ by
$T (\phi)(h) := \phi(h) y_{0}, $ for $ \phi \in \ell_{2}^{\ast}=\ell_{2} $ and $ h \in \ell_{2}. $\ Then
$$ \parallel T(\phi) \parallel=\sup_{h\in B_{\ell_{2}}}\parallel T(\phi) (h) \parallel=\sup_{h\in B_{\ell_{2}}}\parallel \phi(h) y_{0} \parallel=\parallel \phi \parallel .$$ 
Since, uniformly $ p $-convergent sets are stable under isometry,
there exists a uniformly $ 1 $-convergent subset of $ L(\ell_{2},Y) $ such  that 
it is not uniformly $ 2 $-convergent.\
\end{ex}
In the following result, we give some  properties of uniformly $ p $-convergent sets.
\begin{prop}\label{p1} 
$ \rm{(i)} $ Every subset of a uniformly  $ p $-convergent set in $ L(X,Y) $ is uniformly $ p $-convergent.\\
$ \rm{(ii)} $ Absolutely closed convex hull of a uniformly $ p $-convergent set in $ L(X,Y) $ is uniformly $ p $-convergent.\\
$ \rm{(iii)} $ If $ K_{1},\cdot \cdot\cdot, K_{n} $ are uniformly $ p $-convergent sets in $ L(X,Y), $ then $ \displaystyle\bigcup_{i=1}^{n} K_{i}$ and $ \displaystyle \sum_{i=1}^{n} K_{i} $ are uniformly $ p $-convergent sets in $ L(X,Y). $\\ 
$ \rm{(iv)} $ Every  relatively compact subset of
$C_{p}(X, Y ) $ 
is  uniformly $ p $-convergent.\\
$ \rm{(v)} $
If $ K \subset L(X, Y ) $ is a uniformly $ p $-convergent set, then $ K \subset C_{p}(X, Y ) .$\
\end{prop}

\begin{thm}\label{t1}
Let $ X $ be a Banach space.\ If there exists a Banach space $  Y$ so that every uniformly $ p $-convergent
 set of $  K(X, Y )  $ is equicompact, then $ C_{p}(X, Y ) =K(X,Y). $ 
\end{thm}
\begin{proof} Since the $ p $-$ (V) $ sets in $ X^{\ast} $ coincides with the uniformly $ p $-convergent  subsets of $ X^{\ast}, $ it is enough to show that every uniformly $ p $-convergent  subset $  M$ of $ X^{\ast} $ is relatively
compact; see {\rm (\cite[Theorem 2.4]{afz})}.\ For this purpose, consider $ y_{0}\in S_{Y} $ and put $ H = M\otimes y_{0}. $\ Obviously, $  H$
is a uniformly $ p $-convergent subset of $ K(X, Y ). $\ Hence, by the hypothesis,
$ H $ is equicompact, which yields the equicompactness of $  M$ as a subset of
$ K(X, \mathbb{R}) .$\  Therefore, an  application of {\rm (\cite[Lemma 2.1]{spd1})} yields the result.
\end{proof}
As an immediate consequence of Proposition 2.2 of \cite{spd}, one can
conclude the following result.
\begin{prop}\label{p2}
If $  B_{X}$ is weakly $ p $-precompact, then a bounded subset $ M$ of $ K(X, Y )$ is  equicompact if and only if
 $ M $ is uniformly $ p $-convergent.
\end{prop}

\begin{ex}\label{e2}
For each $ b =(b_{n})_{n}\in \ell_{2}, $
 Define the operator $ T_{b} :\ell_{2}\rightarrow \ell_{1}$ by $ T_{b} (a_{n}):=(a_{n}b_{n}).$\ Obviously,
 $  M:=\lbrace  T_{b} :b\in B_{\ell_{2} }  \rbrace$ is  not equicompact in $ K(\ell_{2},\ell_{1}). $\ So, Proposition $ \rm{\ref{p2}} $ implies that $ M $ is not a uniformly $ p $-convergent
subset of $  K(\ell_{2}, \ell_{1}).$ 
\end{ex}
A subset $ M $ of $ K(X, Y )$ is said to be collectively
compact, if $ \bigcup_{T\in M} T(B_{X}) $ is a relatively compact set.\ Recall that $ M \subset K(X, Y ) $ is equicompact
 if and only if $ M^{\ast} = \lbrace T^{\ast} : T \in M\rbrace $ is collectively
compact; see \cite{spd}.
\begin{prop}\label{p3} 
Let $ S: X\rightarrow Z $ be a weakly $ p $-precompact operator.\ If for each Banach space $  Y,$ every 
$ N \subset C_{p}(Z, Y ) $ is uniformly $ p $-convergent, then the set 
 $ N\circ S:=\lbrace T\circ S: T \in N\rbrace $ is equicompact.
\end{prop}
\begin{proof}
 We prove that $ S^{\ast}\circ N^{\ast} $ is collectively compact.\ Consider  a sequence $( (S^{\ast}\circ T^{\ast}_{n})y^{\ast}_{n}) _{n}$
in $ \bigcup_{T\in N} S^{\ast}\circ T^{\ast}(B_{Y})$ and put $A:=\lbrace T^{\ast}_{n} y^{\ast}_{n}:n\in \mathbb{N}    \rbrace  .$\ It is easy to verify that, $ A $ is a uniformly $ p $-convergent set in $ Z^{\ast} .$\ Indeed, if $ (z_{n})_{n}\in \ell_{p}^{w} (Z),$ we have
$$\lim_{n\rightarrow \infty}\sup_{m} \vert \langle z_{n} , T^{\ast}_{m}(y^{\ast}_{m})     \rangle\vert \leq \lim_{n\rightarrow \infty}\sup_{m}\parallel T_{m} (z_{n}) \parallel=0.$$
Let $ (z^{\ast}_{n}) _{n}\subset A$ and let $ (z_{n})_{n}\in \ell^{w}_{p}(Z) .$\  Cosider an operator $ S_{1}:Z\rightarrow \ell_{\infty} $ defined by 
$ S_{1}(z) :=(z^{\ast}_{n}(z)).$\  
Since $ A $ is uniformly $ p $-convergent set in $ Z^{\ast} ,~ $ $ \lim_{n}\parallel S_{1}(z_{n})  \parallel =\lim_{n}\sup_{i}\vert z^{\ast}_{i} (z_{n}) \vert=0,$ and so $ S_{1} $ is $ p $-convergent.\ Hence, the operator $ S_{1} S:X \rightarrow \ell_{\infty}$ is compact, since 
$ S: X\rightarrow Z $ is a weakly $ p $-precompact operator.\ Thus $ S^{\ast}\circ S^{\ast}_{1} $ is compact and so, $ S^{\ast}(z^{\ast}_{n})_{n}=(S^{\ast}(S^{\ast}_{1}(e^{1}_{n}))_{n} $ is
relatively compact, where $ (e^{1}_{n}) $ is the unit basis of $ \ell_{1} .$\ Hence,
 $ S^{\ast}(A)$ is a relatively compact set and so,
$( (S^{\ast}\circ T^{\ast}_{n})y^{\ast}_{n}) _{n}$ has a convergent subsequence.\
\end{proof}

In \cite{spd1},
the authors defined weakly equicompact sets as those subsets $ M $ of $ W(X, Y ) $ satisfying that, for
every bounded sequence $ (x_{n})_{n} $ in $ X, $ there exists a subsequence $ (x_{k_{n}})_{n} $ such that $ (T(x_{k_{n}}))_{n} $ is
weakly uniformly convergent for $ T \in M. $\
\begin{prop}\label{p4} 
 If $ B_{X} $ is weakly $ p $-precompact, then  the following statements are equivalent
for a set $   M \subset W(X, Y ) .$\\
$ \rm{(i)} $
$ M $ is weakly equicompact.\\
$ \rm{(ii)} $ $ M^{\ast}(y^{\ast}) $ is a uniformly  $ p $-convergent set in $ X$ for every $ y^{\ast} \in Y^{\ast}.$\\
$ \rm{(iii)} $  $ {T(x_{n})\stackrel []{w}\rightarrow 0} $
 uniformly for $ T \in M  $ whenever $ (x_{n})_{n}\in \ell_{p}^{w} (X).$
\end{prop}
\begin{proof} 
Since the assertions  (i) $ \Rightarrow $ (ii) is straightforward of {\rm (\cite[Corollary 2.3]{spd1})} and  (ii) $ \Rightarrow $ (iii)  is obvious, we only have to show that  (iii) $ \Rightarrow $ (i).\\ 
Let $ (x_{n})_{n} $ be a bounded sequence in $ X. $\ Since $ B_{X} $ is weakly $ p $-precompact, we can suppose that $ (x_{n})_{n} $ is weakly $ p $-Cauchy.\ Assume that $ M $ is not weakly equicompact.\ Thus, $ (T(x_{n})) _{n}$ is not weakly Cauchy uniformly for $ T \in M. $\ So, there exist $ \varepsilon >0,~ y^{\ast}\in Y^{\ast} ,$ 
strictly increasing sequences $ (p_{n})_{n} \subset \mathbb{N}$
and $ (q_{n}) _{n}\subset \mathbb{N},$ a sequence $ (T_{n})_{n} $ in $ M $ such that:
\begin{center}
$\vert  \langle  x_{p_{n}} -x_{q_{n}} , T_{n}^{\ast} (y^{\ast}) \rangle\vert=\vert  \langle  T_{n}(x_{p_{n}}) -T_{n}(x_{q_{n}}) ,  y^{\ast} \rangle\vert \geq \frac{\varepsilon}{2} ~~~~~~$ for  all $n\in \mathbb{N},$
\end{center}
which is a contradiction.
\end{proof}
In the following example, we show that the hypothesis
about $ X $ cannot be omitted in Proposition \ref{p4}.\
\begin{ex}\label{e3}
Suppose $  M:=B_{\ell_{\infty}}\otimes B_{\ell_{1}}.$\ An easy verification
shows that $M\subset W(\ell_{1},\ell_{1})$ 
satisfies conditions $\rm{ (ii) }$ and  $\rm{ (iii) }$  of   Proposition $ \rm{\ref{p4}} $ while, $ M $ is not weakly equicompact.\
\end{ex}
Here, we obtain a characterization of double dual of
 Banach space $ X $ with the $ p $-Schur property.\ 
\begin{thm}\label{t2}
If $ X $ is a Banach space, then the following statements are equivalent:\\
$ \rm{(i)} $ For every Banach space $ Y ,$ if $ M \subset W(X^{\ast}, Y )  $ is relatively weakly $ p $-compact, then
it  is weakly equicompact.\\
$ \rm{(ii)} $ For some Banach space $ Y\neq \lbrace 0\rbrace, $ if $ M \subset W(X^{\ast}, Y )  $ is relatively weakly $ p $-compact, then
it  is weakly equicompact.\\
$ \rm{(iii)} $ $ X^{\ast\ast} $
 has the $ p $-Schur property.
\end{thm}
\begin{proof} 
The assertion  (i) $ \Rightarrow $ (ii) is  straightforward.\ Therefore,  we only  prove that the assertions  (ii) $ \Rightarrow $ (iii) and  (iii) $ \Rightarrow $ (i).\\
 (ii) $ \Rightarrow $ (iii) Let  $ K $ be a relatively weakly $ p $-compact set
 in $  X^{\ast\ast}. $\ We claim that $ K $ is relatively  norm compact.\ For this purpose, consider
  $ M=K\bigotimes y_{0},$ so that  $y_{0} \in Y- \lbrace 0\rbrace. $\ One can see that, $ M  $ is a relatively weakly $ p $-compact set in $ W(X^{\ast},Y). $\ By the hypothesis, $ M $ is weakly equicompact, which yields the
weakly equicompactness of $ K .$\ Hence, {\rm (\cite[Lemma 2.1]{spd1})} implies that $ K $ is relatively  norm compact.\\
  (iii) $ \Rightarrow $ (i) Suppose  that $ M \subset W(X^{\ast}, Y ) $ is relatively weakly $ p $-compact and $ (x^{\ast}_{n}) _{n}\in \ell_{p}^{w}(X^{\ast}).$\ If  $ {T(x^{\ast}_{n})\stackrel []{w}\rightarrow 0} ,$
so that the convergence is not
uniform for $ T \in M,$ then there exist $ y^{\ast} \in Y^{\ast},~ \varepsilon>0,$
 strictly increasing sequences $ (p_{n}) _{n}\subset \mathbb{N}$
and $ (q_{n})_{n}\subset \mathbb{N} ,$ and a sequence $ (T_{n}) _{n}$ in $  M$ such that:
\begin{center}
$\vert  \langle  x^{\ast}_{p_{n}} -x^{\ast}_{q_{n}} , T_{n}^{\ast}( y^{\ast}) \rangle\vert=\vert  \langle  T_{n}(x^{\ast}_{p_{n}}) -T_{n}(x^{\ast}_{q_{n}}) ,  y^{\ast} \rangle\vert \geq \frac{\varepsilon}{2} ~~~~~~$ for  all $n\in \mathbb{N},$
\end{center}
On the other hand, $ (x^{\ast}_{p_{n}}-x^{\ast}_{q_{n}})_{n} $ is weakly $ p $-Cauchy
and $ (T^{\ast}_{n} (y^{\ast})) _{n}$
 admits a weakly convergent subsequence.\ Since $  X^{\ast\ast}$
 has the $ p $-Schur property, Theorem 2.8 of \cite{afz} implies that 
 $  X^{\ast\ast}\in (DPP_{p}).$\ Hence, by using Theorem 3.1 in \cite{ccl}, we have $ \displaystyle\lim_{n\rightarrow \infty} \vert  \langle  x^{\ast}_{p_{n}} -x^{\ast}_{q_{n}} , T_{n}^{\ast} (y^{\ast}) \rangle\vert=0,$ which is a contradiction.
\end{proof}

\begin{prop}\label{p5}
Let $ X $ be a Banach space.\ If there exsits a Banach space $  Y\neq \lbrace 0\rbrace$ such that every uniformly $ p $-convergent
 set of $  W(X, Y )  $ is weakly equicompact, then $ C_{p}(X, Y ) =K(X,Y). $
\end{prop}
\begin{proof}
It suffices to show that every
uniformly $ p $-convergent set $ K\subset X^{\ast} $
is relatively compact; see {\rm (\cite[Theorem 2.4]{afz})}.\ Choose $ y_{0} \in Y$
 and $ y^{\ast}_{0}\in Y^{\ast} $ such that $ \langle y^{\ast}_{0} , y_{0} \rangle =1.$\ Clearly, $ M=K\bigotimes y_{0} $ is a
uniformly $ p $-convegent set in $ W(X,Y) $  and so, by the hypothesis, $ M $ is weakly equicompact.\ Hence, by using
Proposition 2.2 of \cite{spd1},
$ K =\langle y^{\ast}_{0} , y_{0} \rangle K =M^{\ast}(y^{\ast}_{0})$ is relatively compact.
\end{proof}

\begin{defn} \label{d2} Let $ U\subset X$ be an open convex and $ 1 \leq p \leq \infty. $\
We say that $ f: U\rightarrow Y $ is a weakly $ p $-sequentially continuous map, if it takes $ U $-bounded and weakly $ p $-Cauchy  sequences of $ U $ into norm convergent
sequences in $ Y. $\ We denote the space of all
 such mappings by $ C_{wsc}^{p}(U, Y ). $
\end{defn}
 The class of all weakly $ \infty $-sequentially continuous mappings is precisely the class of all weakly sequentially continuous mappings; see \cite{ce2}.\ Also,
note that
 $ C^{q}_{wsc} (U,Y)\subseteq C^{p}_{wsc} (U,Y), $ whenever $ 1\leq p<q \leq \infty. $\ But, we do not have any example of a mapping $ f \in C^{1u}(U, Y ) \cap C^{p}_{wsc}(U, Y ) $ which does not belong to $ C^{q}_{wsc}(U, Y ). $\

\begin{prop}\label{p6} Let $ U\subset X$ be an open convex and $ 1 \leq p \leq \infty. $\
If $ f $ is compact and takes $ U $-bounded and weakly $ p $-Cauchy sequences into weakly Cauchy sequences, then $ f \in C_{wsc}^{p}(U, Y ). $ 
\end{prop}
\begin{prop}\label{p7} Let $ U\subset X$ be an open convex and $ 1 \leq p \leq \infty. $\
If $ f \in C^{1u}(U, Y )$ so that  $ f^{\prime} \in C^{p}_{wsc} (U, C_{p}(X,Y)),$ then $ f\in C_{wsc}^{p} (U,Y).$
\end{prop}
\begin{proof}
Let $ (x_{n})_{n} $ be a $ U $-bounded and weakly $ p $-Cauchy sequence.\ By the Mean Value Theorem {\rm (\cite[Theorem 6.4]{ch})}, we have
$$ \parallel f(x_{n})-f(x_{m}) \parallel ~\leq ~\parallel f^{\prime}(c_{n,m}) (x_{n}-x_{m})\parallel$$
for some $ c_{n,m} \in I(x_{n}, x_{m}). $\ Since the sequence $ (c_{n,m}) $ is $ U$-bounded and weakly $ p $-Cauchy, the sequence $ (f^{\prime}(c_{n,m})) $ norm
converges
to some $ T \in C_{p}(X, Y ). $\ Therefore we have:\\ $$ \displaystyle\lim_{n,m\rightarrow \infty} \parallel f^{\prime}(c_{n,m}) (x_{n}-x_{m})-T(x_{n}-x_{m})\parallel+T(x_{n}-x_{m}) \parallel $$\\ $$ \leq \lim_{n,m\rightarrow \infty} \parallel f^{\prime}(c_{n,m}) (x_{n}-x_{m})-T(x_{n}-x_{m})\parallel+\lim_{n,m\rightarrow\infty}\parallel T(x_{n}-x_{m}) \parallel=0.$$
So,
$ \displaystyle\lim_{n,m\rightarrow \infty}\parallel f^{\prime}(c_{n,m}) (x_{n}-x_{m})\parallel=0.$
Hence, the sequence $ (f(x_{n}))_{n} $ is norm convergent.
\end{proof}
\begin{prop}\label{p8}  Let $ U\subset X$ be an open convex and $ 1 \leq p \leq \infty. $\
If $ f:U\rightarrow Y $ is a differentiable mapping such that for every $  U$-bounded  set
$ K, $ $ f^{\prime}(K) $ is a  uniformly $ p $-convergent set in $ L(X, Y ),$\ then $ f\in C_{wsc}^{p} (U,Y). $
\end{prop}
\begin{proof} 
Let $ (x_{n})_{n} $ be a $ U $-bounded and weakly $ p $-Cauchy sequence.\ By the Mean Value Theorem {\rm (\cite[Theorem 6.4]{ch})}, for all $ n,m\in \mathbb{N} ,$ there is
$ c_{i,j} \in I(x_{n},x_{m})$ such that
$$ \parallel f(x_{n}) -f(x_{m}) \parallel \leq \parallel f^{\prime}(c_{i,j})(x_{n}-x_{m}) \parallel \leq \sup_{i,j} \parallel f^{\prime}(c_{i,j})(x_{n}-x_{m}) \parallel$$
Obviously, the set $ K:=\lbrace c_{i,j}:i,j \in \mathbb{N} \rbrace $ is contained in the convex hull of all $ x_{n} $ and then in
$ U, $ since $ U $ is a convex set.\ Moreover $ K$ is still a  $ U $-bounded set.\
By the hypothesis, $ f^{\prime}(K) $ is a uniformly $ p $-convergent set in $ L(X,Y) $.\ Since $ (x_{n}-x_{m})\in \ell_{p} ^{w}(X),$ it follows that
$ \displaystyle\lim_{_{n,m}}\sup_{i,j}\parallel f^{\prime}(c_{i,j})(x_{n}-x_{m}) \parallel =0. $
Therefore, $ \Vert f(x_{n})-f(x_{m}) \Vert\rightarrow 0. $
\end{proof}
Now by using the same argument of {\rm (\cite[Theorem  2.1]{ce2})},  we  find a method 
to get    uniformly $ p $-convergent subsets of $ L(X, Y ). $
\begin{thm}\label{t3}
Let $ U\subseteq X $ be an open convex subset and $ 1\leq p< \infty. $\ If
$ f \in C^{1u}(U, Y ),$ then the following assertions are equivalent:\\
$ \rm{(i)} $ $ f\in C_{wsc}^{p} (U,Y);$\\
$ \rm{(ii)} $ For every $ U $-bounded and weakly $ p $-Cauchy sequence $ (x_{n}) $ and every weakly $ p $-Cauchy sequence $ (h_{n}) \subset X, $ the sequence $ (f^{\prime}(x_{n})(h_{n}))_{n} $ norm converges
in $ Y; $\\
$ \rm{(iii)} $ For every $ U $-bounded weakly $ p $-Cauchy sequence $ (x_{n})_{n} $ and every weakly $ p $-summable sequence $ (h_{n})_{n} \subset X, $ we have
$$\lim_{n}\sup_{m}\parallel f^{\prime}(x_{m}) (h_{n}) \parallel =0;$$\
$ \rm{(iv)} $ For every $ U$-bounded and weakly $ p $-Cauchy sequence $ (x_{n})_{n} $ and every weakly $ p $-summable sequence $ (h_{n})_{n} \subset X, $ we have
$$\lim_{n} f^{\prime}(x_{n}) (h_{n}) =0;$$
$ \rm{(v)} $ $ f^{\prime} $ takes $ U$-bounded and weakly $ p $-precompact subsets of $ U$ into uniformly $ p $-convergent subsets of $ L(X, Y ). $
\end{thm}
\begin{proof}
(i) $ \Rightarrow $ (ii) Let $ (x_{n})_{n} $ be a $ U $-bounded  weakly $ p $-Cauchy sequence and let $ (h_{n})_{n} $ be a weakly $ p $-Cauchy sequence in $ X. $\ Without loss
of generality, we assume that $ (h_{n})_{n}$ is bounded.\ Consider $ B := \lbrace x_{n}: n \in \mathbb{N} \rbrace $ and let $ d := min\lbrace 1, dist(B, \partial U)\rbrace. $\ It is easy to show that the set
$$ B^{\prime}:=B+\frac{d}{2}B_{X}\subset U$$
is also $ U $-bounded.\ Since $ f \in C^{1u}(U, Y ),$ $ f^{\prime} $ is uniformly continuous on $ B^{\prime}.$\ Hence, for given $ \varepsilon >0, $ there exists $ 0 <\delta <\frac{d}{4} $ such that if $t_{1} ,t_{2}\in B^{\prime}$
satisfy $ \parallel t_{1} -t_{2}\parallel <2\delta,$ then
\begin{equation}\label{eq}
\parallel f^{\prime}(t_{1})-f^{\prime}(t_{2}) \parallel <\frac{\varepsilon}{4}.
\end{equation}
If $ c\in I(x_{n},x_{n}+\delta h_{n}) $ for some $ n\in \mathbb{N}, $ then
$$ \parallel c-x_{n} \parallel \leq \delta \parallel h_{n} \parallel <\delta <2\delta <\frac{d}{2},$$
and so,
$$c=x_{n} +(c-x_{n})\in B^{\prime}=B+\frac{d}{2}B_{X} $$
As an immediate consequence of the Mean Value Theorem {\rm (\cite[Theorem 6.4]{ch})}, and formula (\ref{eq}), we obtain
\begin{flushleft}
$ \parallel f^{\prime}(x_{n}) ( \delta h_{n})-f(x_{n}+
\delta h_{n})+f(x_{n})\parallel $
\end{flushleft}
\begin{center}
$ \leq \displaystyle\sup_{c\in I(x_{n},x_{n}+\delta h_{n})} \parallel f^{\prime} (c) - f^{\prime} (x_{n}) \parallel \parallel\delta h_{n}\parallel
\leq\frac{\varepsilon \delta}{4} .$
\end{center}
Similarly,
\begin{flushleft}
$ \parallel f(x_{m}+\delta h_{m})-f(x_{m})- f^{\prime}(x_{m}) ( h_{m})\parallel $
\end{flushleft}
\begin{center}
$
\leq \displaystyle\sup_{c\in I(x_{m},x_{m}+\delta h_{m})} \parallel f^{\prime} (c) - f^{\prime} (x_{m}) \parallel \parallel\delta h_{m}\parallel
\leq\frac{\varepsilon \delta}{4}.\
$
\end{center}
On the other hand, the sequences $ (x_{n}+\delta h_{n})_{n} $ and $ (x_{n})_{n} $ are $ U$-bounded and weakly $ p $-Cauchy in $ U. $\ Hence, by the hypothesis the
sequences $ ( f (x_{n} +\delta h_{n}))_{n} $ and $ ( f (x_{n}))_{n} $ are norm convergent in $ Y. $\ Hence, we can find $ n_{0}\in \mathbb{N} $ so that for $ n,m > n_{0} :$
\begin{flushleft}
$ \parallel f(x_{n}+\delta h_{n})-f(x_{m}+\delta h_{m})\parallel <\frac{\varepsilon \delta}{4},~~~~~~~~~~~~~ \parallel f(x_{n})-f(x_{m})\parallel <\frac{\varepsilon \delta}{4} $
\end{flushleft}
So, for $ n,m > n_{0} ,$ we have\
$$\parallel f^{\prime}(x_{n}) (h_{n}) - f^{\prime}(x_{m}) (h_{m})\parallel< \varepsilon.$$\
(ii) $ \Rightarrow $ (iii) Let $ (x_{n})_{n} $ be a $ U $-bounded weakly $ p $-Cauchy sequence and let $ (h_{n})_{n} $ be a weakly $ p $-summable sequence in $ X. $\ By the part $ \rm{(ii)} $, for every $ h \in X, $ the set $ \lbrace f^{\prime} (x_{n})(h) :n\in \mathbb{N}\rbrace$ is bounded in $ Y. $\ On the other hand, there exists a subsequence $ (x_{m_{k}}) _{k}$ of $ (x_{m})_{m} $ in $ U $ such that
$$ \parallel f^{\prime} (x_{m_{k}})(h_{k}) \parallel \geq \sup_{m}\parallel f^{\prime} (x_{m})(h_{k}) \parallel-\frac{1}{k} ~~~~(k\in \mathbb{N}).$$
Since the sequences $ (x_{m_{k}}) _{k}$ in $ U $ and $ (h_{1},0,h_{2},0,h_{3},0,\cdot\cdot\cdot) $ in $ X $
are weakly $ p $-Cauchy, the sequence
$$ (f^{\prime}(x_{m_{1}})(h_{1}),0,f^{\prime}(x_{m_{2}})(h_{2}),0, f^{\prime}(x_{m_{3}})(h_{3}),0,\cdot\cdot\cdot) $$
converges in $ Y. $\ Therefore, $ \displaystyle\lim_{k} f^{\prime}(x_{m_{k}})(h_{k})=0.$\ Hence, we have:
$$ \lim_{k} \sup_{m}\parallel f^{\prime}(x_{m}) (h_{k}) \parallel =0.$$
(iii) $ \Rightarrow $ (iv) is obvious.\\
(iv) $ \Rightarrow $ (v) Let $ K$ be a weakly $ p $-precompact and $ U $-bounded set.\ It is clear that, for every $ h \in X, $ the set $ f^{\prime}(K)(h) $
is bounded
in $ Y. $\ Let $ (h_{n}) _{n}$ be a weakly $ p $-summable sequence in $ X. $\
If $ (h_{n_{k}})_{k} $ is a subsequence of $ (h_{n}) _{n},$ then for every $ k\in \mathbb{N}, $ there exists $ a_{k}\in K $ such that
$$ \sup_{a \in K}\parallel f^{\prime}(a) (h_{n_{k}}) \parallel < \parallel f^{\prime}(a_{k}) (h_{n_{k}}) \parallel+\frac{1}{k}. $$
Since $ K $ is a weakly $ p $-precompact set, the sequence $ (a_{k})_{k} $ admits a weakly $ p $-Cauchy subsequence $ (a_{k_{r}}) _{r}.$\ Hence, by the hypothesis, 
$$ \displaystyle\lim_{r} \parallel f^{\prime}(a_{k_{r}}) (h_{n_{k_{r}}}) \parallel=0. $$
Therefore we have,
$ \displaystyle\lim_{r}\sup_{a\in K} \parallel f^{\prime}(a) (h_{n_{k_{r}}}) \parallel=0. $\ Hence, every subsequence of
$ ( \displaystyle \sup_{a \in K} \parallel f^{\prime}(a) (h_{n}) \parallel) _{n}$
has a subsequence converging to $ 0. $\ Therefore, the sequence itself converges to $ 0, $ that is, $ \displaystyle \lim_{n}\sup_{a \in K} \parallel f^{\prime}(a) (h_{n}) \parallel =0.$\\
(v) $ \Rightarrow $ (i) Let $ (x_{n})_{n} $ be a $ U $-bounded and weakly $ p $-Cauchy sequence.\ Since $ U$ is convex, the segment $ I(x_{n}, x_{m}) $ is contained in $ U$
for all $ n,m \in \mathbb{N}. $\ By the Mean Value Theorem {\rm (\cite[Theorem 6.4]{ch})},
there exists $ c_{nm} \in I(x_{n}, x_{m}) $ such that
$$ \parallel f(x_{n}) -f(x_{m}) \parallel \leq \parallel f^{\prime}(c_{n,m}) (x_{n}-x_{m}) \parallel \leq \sup_{i,j \in \mathbb{N}}\parallel f^{\prime}(c_{i,j}) (x_{n}-x_{m}) \parallel .$$
Since $ (c_{i,j})_{i,j} $ is a weakly $ p $-Cauchy and $ U $-bounded sequence, the part $\rm{(v)}$ implies that
$$ \lim_{n,m}\sup_{i,j \in \mathbb{N}}\parallel f^{\prime}(c_{i,j}) (x_{n}-x_{m}) \parallel=0. $$
Therefore, $\displaystyle \lim_{n,m\rightarrow\infty}\parallel f(x_{n}) -f(x_{m}) \parallel =0.$
\end{proof}
\begin{ex}\label{e5} 
Let $ h\in C^{1}(\mathbb{R}) $ and $ 1<p<2. $\ Define $ f:\ell_{p^{\ast}}\rightarrow \mathbb{R} $ by 
$ f((x_{n})_{n}) = \displaystyle\sum _{n=1}^{\infty}\frac{h(x_{n})}{2^{n}}.$\
The same argument as in the {\rm (\cite[Example 2.4]{ce4})}, shows that $ f $ is differentiable so that
$ f^{\prime} ((x_{n})_{n})=(\frac{h^{\prime}(x_{n})}{2^{n}})_{n} \in \ell_{p}.$\  By Pitt's Theorem {\rm (\cite[Theorem 2.1.4]{ce4})},
$ f^{\prime}:\ell_{p^{\ast}} \rightarrow \ell_{p}$ is compact and so, $ f^{\prime}(B_{\ell_{p^{\ast}}}) $ is a relatively compact set in $L(\ell_{p^{\ast}},\mathbb{R})= C_{p} (\ell_{p^{\ast}},\mathbb{R}).$\ Therefore,  the part $ \rm{(iv)} $ of Proposition $ \rm{\ref{p1}} $, yields  that $ f^{\prime}(B_{\ell_{p^{\ast}}}) $ is a uniformly $ p $-convergent set in $ L(\ell_{p^{\ast}}, \mathbb{R}) .$\ 
Hence, 
 Proposition $ \rm{ \ref{p8}} $ implies that $ f$ is weakly $ p $-sequentially continuous.\ On the other nand, $ \frac{1}{2} B_{\ell_{p^{\ast}}} $ is a $ B_{\ell_{p^{\ast}}} $-bounded set.\ Hence, by Theorem $ \rm{\ref{t3}} $ $ f^{\prime} (\frac{1}{2} B_{\ell_{p^{\ast}}}) $ is a $ p $-(V) set in $ \ell_{p}. $
\end{ex}
\section{Factorization theorem through a $ p $-convergent operator}
Here, 
for given a  mapping $ f : U \rightarrow Y ,$  we show that $ f $ is differentiable so that $ f^{\prime} $ takes $ U $-bounded sets
into uniformly $ p $-convergent sets if and only if it happens $ f=g\circ S, $ where $  S$ is a  $ p $-convergent
 operator  from $ X $ into a suitable Banach space $ Z$ and $ g : S(U) \rightarrow Y $  is a G$ \hat{a} $teaux differentiable mapping with some additional properties.\
\begin{thm}\label{t4}  Let $ U\subset X$ be an open convex and $ 1 \leq p \leq \infty. $\
If $ f:U\rightarrow Y $ is a mapping,
then the following assertions are equivalent:\\
$ \rm{(a)} $ $ f $ is differentiable so that $ f^{\prime} $ takes $ U $-bounded sets into uniformly $ p $-convergent sets and $ f$ is weakly $ p $-sequentially continuous.\\
$ \rm{(b)} $ There exist a Banach space $ Z,$ an operator $ S\in C_{p}(X, Z ) $ and a mapping
$ g:S(U) \rightarrow Y$ such that:\\
$ \rm{ (i)} $ $ f(x) = g(S(x)) $ for all $ x \in U.$\\
$ \rm{ (ii)} $ $ g \in D_{\mathcal{M}}(S(x), Y ) $ for every $ x \in U, $ where
\begin{center}
$ \mathcal{M}:=\lbrace S(B) :B $ is a $ U $-bounded subset of $ X \rbrace .$
\end{center}\
$ \rm{ (iii)} $ $ g^{\prime} $ is bounded on $ S(B) $ for every $ U $-bounded subset $ B\subset X.$\\
Moreover, if this factorization holds, $ f $ is weakly $p $-sequentially continuous.
\end{thm}
\begin{proof}
(a) $ \Rightarrow (b) $ For every $ n\in \mathbb{N}, $ put
$$ W_{n}=\lbrace x \in U: d(x, \partial U)>\frac{1}{n} \rbrace \bigcap nB_{X} .$$
By the hypothesis for every $ r\in \mathbb{N}, $ $ f^{\prime}(\frac{W_{r}}{r \parallel f^{\prime}\parallel_{W_{r}}}) $ is a uniformly $ p $-convergent set.\  Now, we define $ K:=\displaystyle\bigcup _{r=1}^{\infty} \frac{f^{\prime}(W_{r})}{r\parallel f^{\prime}\parallel_{W_{r}}}.$\ We claim that, $ K $ is a uniformly $ p $-convergent set.\ Indeed, for every $ N\in \mathbb{N}, $ we define $ A_{N} := \displaystyle\bigcup_{r\leq N}\frac{f^{\prime}(W_{r})}{r \parallel f^{\prime} \parallel}_{W_{r}}$ and $ B_{N} := \displaystyle\bigcup_{r> N}\frac{f^{\prime}(W_{r})}{r \parallel f^{\prime} \parallel}_{W_{r}}.$\ Since
$ A_{N} $ is a uniformly $ p $-convergent set, it is enough to show that
$ B_{N} $ is a uniformly $ p $-convergent set.\ For this purpose, let $ (x_{n})_{n}\in \ell_{p}^{w} (X)$  and $ M= \displaystyle\sup _{n} \parallel x_{n} \parallel.$\ If $ T \in B_{N} ,$ then there are $ r>N $ and $ x\in W_{r} $ so that $ T=\frac{f^{\prime}(x)}{r\parallel f^{\prime} \parallel_{W_{r}}} .$\ It is clear that, $ \parallel T\parallel\leq\frac{1}{r} <\frac{1}{N}.$\ Hence, for each $ N\in \mathbb{N} $ we have: 
$$ \lim_{n\rightarrow \infty}\sup_{T\in K}\parallel T(x_{n}) \parallel =max\lbrace \lim_{n\rightarrow \infty}\sup_{T \in A_{N}} \parallel T(x_{n})\parallel ,\lim_{n\rightarrow \infty}\sup_{T \in B_{N}} \parallel T(x_{n})\parallel \rbrace\leq \frac{M}{N} .$$
Therefore, $ \displaystyle \lim_{n\rightarrow \infty}\sup_{T\in K}\parallel T(x_{n}) \parallel =0 $ and so, $ K $ is a uniformly $ p $-convergent set.
Now,
as in the proof of {\rm (\cite[Theorem 2.1]{ce3})}, let
$$ V_{K}:=\lbrace x\in X: \displaystyle\sup _{\phi\in K}\parallel \phi(x) \parallel _{Y}=0 \rbrace $$
and $ G:=\frac{X}{V_{k}} .$\ If $ S:X\rightarrow G $ is the quotient map $ G, $ then $ G $ is normed space
respect the norm $ \parallel S(x) \parallel =\displaystyle\sup _{\phi\in K}\parallel \phi(x) \parallel _{Y}, ~~~\forall x \in X.$\ Suppose that $ Z $ is the completion of $ G .$\ Let $ (x_{n}) _{n}\in \ell_{p}^{w}(X),$ since $ K$ is a uniformly $ p $-convergent set, $ \parallel S(x_{n}) \parallel =\displaystyle\sup _{\phi\in K}\parallel \phi(x_{n}) \parallel\rightarrow 0. $\ Hence, $ S\in C_{p} (X,Z).$\
Now we define $ g : S(U) \rightarrow Y $ by $ g(S(x))=f(x) ,~~~~ x\in U.$\ In the first, we proved that $ g $ is well defined.\ Suppose that $ \parallel S(x-y) \parallel=0. $\ Since the span of $ K$ contains $ f^{\prime}(U) ,$ we have
$$ \parallel f^{\prime}(c) (x-y) \parallel =0 ~~~ ~~~~~(c \in U).$$
By using the Mean Value Theorem {\rm (\cite[Theorem 6.4]{ch})},
$$\displaystyle\sup _{c\in I(x,y)}\parallel f(x)- f(y) \parallel\leq \parallel f^{\prime}(c) (x-y) \parallel =0,$$ and so $ f(x)=f(y) .$\ Therefore $g$ well defined.\ Now, we show that $ g $ is G$ \hat{a} $teaux differentiable.\ For given $ x,y \in U, $
\begin{equation}\label{eqeq}
\displaystyle\lim_{t\rightarrow 0}\frac{g(S(x)+tS(y))-g(S(x))}{t}=\lim_{t\rightarrow 0}\frac{f(x+ty)-f(x)}{t} =f^{\prime}(x)(y)
\end{equation}
where $\mid t \mid $ is sufficiently small so that $ x + ty \in U. $\ For $ x\in U $ fixed, the mapping
$ g^{\prime}(S(x)) :G\rightarrow Y$ given by $g^{\prime}(S(x)) (S(y)) =f^{\prime}(x)(y) ~~~~~(y \in X) $
is linear.\ Choosing $ r \in \mathbb{N} $ so that $ x \in W_{r}, $ we have
\begin{eqnarray*}
\parallel g^{\prime}(S(x)) (S(y)) \parallel &=&\\ \parallel f^{\prime}(x)(y) \parallel
&\leq & r\parallel f^{\prime}\parallel_{W_{r}} \parallel S(y) \parallel.
\end{eqnarray*}
Hence $g^{\prime} (S(x)) $ is continuous and may be extended to the completion $ Z$ of $ G. $\ Hence $ g$ is
G$ \hat{a} $teaux differentiable.\ Moreover, since $ f $ is differentiable, for every $ U $-bounded set
$ B, $ the limit in (\ref{eqeq}) exists uniformly  to $ S(y) $ in $ S(B). $\ So,
$ g \in D_{\mathcal{M}}(S(x), Y ) $ for every $ x \in U, $ where
$$ \mathcal{M}=\lbrace S(B) : B ~ is ~ a~ U-bounded ~ subset ~ of ~X \rbrace $$
and (ii) is proved.\\
From the inequality $\rm{ (ii)}, $ we have
\begin{center}
$ \parallel g^{\prime}(S(x))\parallel =\displaystyle\sup_{\parallel S(y)\parallel \leq1} \parallel g^{\prime} (S(x)) (S(y)) \parallel \leq r\parallel f^{\prime}\parallel_{W_{r}}, ~~(x \in W_{r})$
\end{center}
and this implies (iii).\\
(b) $ \Rightarrow $ (a). Assume that there exists a Banach space $ Z$ and  $ S\in C_{p}(X,Z), $
and a mapping $ g : S(U) \rightarrow Y $ satisfying $ (b). $\ Obviously, $ f$ is differentiable.\ We claim that $ f^{\prime} $ takes $ U $-bounded sets into uniformly $ p $-convergent sets.\ For this purpose, suppose that $ B $ is a
$ U$-bounded set and $ (x_{n})_{n}\in \ell_{p}^{w}(X).$\ Since $ S \in C_{p}(X,Z), $ we have
\begin{eqnarray*}
\displaystyle\sup _{x\in B} \parallel f^{\prime}(x)(x_{n})\parallel = \displaystyle \sup_{x\in B}\parallel g^{\prime}(S(x))(S(x_{n}))\parallel
&\leq &\\ \displaystyle \sup_{x\in B}\parallel g^{\prime}(S(x))\parallel \parallel S(x_{n})\parallel\rightarrow 0.
\end{eqnarray*}
So, $ f^{\prime}(B) $ is a uniformly $ p $-convergent set.\
Given $ r\in \mathbb{N},$ the mapping $ g $ is uniformly continuous on $ S(W_{r}) .$ Indeed, for $ x, y\in W_{r} ,$
we have
\begin{flushleft}
$ \parallel g(S(x))-g(S(y)) \parallel=\parallel f(x) -f(y) \parallel \leq \displaystyle \sup_{c\in W_{r}}\parallel f^{\prime}(c)(x-y) \parallel $
\end{flushleft}
\begin{center}
$ \leq r\parallel f^{\prime}\parallel_{W_{r}} \displaystyle \sup_{\phi\in K }\parallel \phi(x-y)\parallel=r\parallel f^{\prime}\parallel_{W_{r}} \parallel S(x-y)\parallel $
\end{center}
where we have used that $ W_{r} $ is a convex set.\ Hence, if $ (x_{n})_{n} $ is a $ U $-bounded  weakly
$ p $-Cauchy sequence, then the sequence $ (S(x_{n}))_{n}$ in $S(W_{r}), $ for a suitable index
$ r, $ is norm Cauchy and so, $ (f(x_{n}))_{n} = (g(S(x_{n})))_{n} $ is also  norm
Cauchy.\ Hence, $ f$ is weakly $ p $-sequentially continuous.
\end{proof}
\begin{ex}\label{e6} 
Let $ h\in C^{1}(\mathbb{R}) .$\ Define $ f:c_{0}\rightarrow \mathbb{R} $ by 
$ f((x_{n})_{n}) = \displaystyle\sum _{n=1}^{\infty}\frac{h(x_{n})}{2^{n}}.$\
It is easy verify that $ f $ is differentiable so that
$ f^{\prime} ((x_{n})_{n})=(\frac{h^{\prime}(x_{n})}{2^{n}})_{n}\in \ell_{1}.$\ Since
$ f^{\prime}:c_{0} \rightarrow L(c_{0},\mathbb{R})=C_{p}(c_{0},\mathbb{R})$ is compact, $ f^{\prime}(B_{c_{0}}) $ is a relatively compact set in $ C_{p} (c_{0},\mathbb{R})$ and so,  $ f^{\prime}(B_{c_{0}}) $ is a uniformly $ p $-convergent set in $ \ell_{1}.$\ 
Therefore, 
 Proposition $ \rm{ \ref{p8}} $ implies that $ f$ is weakly $ p $-sequentially continuous.\ Hence,  there exist a Banach space $ Z,$ an operator $ S\in C_{p}(c_{0}, Z ) $ and a G$ \hat{a} $teaux differentiable mapping 
$ g:S(B_{c_{0}}) \rightarrow \mathbb{R}$ such that $ f=g\circ S $ with some additional properties.
\end{ex}


\end{document}